\documentclass[12pt,a4paper]{article}
 
\usepackage[utf8]{inputenc}
\usepackage[T1]{fontenc}
\usepackage[english]{babel}
\usepackage{lmodern}
 \usepackage{graphics}
\usepackage{amsfonts}
\usepackage{amsmath}
\usepackage{amssymb}
\usepackage{amsthm}
 
\usepackage{geometry}
\newtheorem{thm}{Theorem}
\newtheorem{cor}{Corollary}
\newtheorem{lem}{Lemma}

\newcommand{\Ocal}{\mathcal{O}}
\newcommand{\TODO}[1]{}
\newcommand{\EXCLUDE}[1]{}
 \providecommand{\Log}{\mathop{\rm Log}\nolimits}
 \newcommand{\mytitle}[1]{\goodbreak\medskip\noindent{\itshape #1}}
\date{19 March 2015}
 
\begin{document}
 
\title{An asymptotic distribution for $\left|L^\prime/L(1,\chi)\right|$}
 
\author{Sumaia Saad Eddin\thanks{S. Saad Eddin is supported by the Austrian Science Fund (FWF): Project F5507-N26, which is a part of the Special Research Program `` Quasi Monte Carlo Methods: Theory and Applications".}}
 
\maketitle

\begin{abstract}
Let $\chi$ be a Dirichlet character modulo $q$, let $L(s, \chi)$ be the attached Dirichlet $L$-function, 
and let $L^\prime(s, \chi)$ denotes its derivative with respect to the complex variable $s$. The main purpose of this paper is to give an asymptotic formula for the $2k$-th power mean value of $\left|L^\prime/L(1, \chi)\right|$ when $\chi$ ranges a primitive Dirichlet character modulo $q$ for $q$ prime. We derive some consequences, in particular a bound for the number of $\chi$ such that $\left|L^\prime/L(1, \chi)\right|$ is large. 
\end{abstract}
 
 
\section{Introduction and results} 
Let $\chi$ be a Dirichlet character modulo $q$, let $L(s, \chi)$ be the attached Dirichlet $L$-function, 
and let $L^\prime(s, \chi)$ denotes its derivative with respect to the complex variable $s$. The values at $1$ of Dirichlet $L$-series has received considerable attention, 
due to their algebraical or geometrical interpretation. Let us mention, in particular, the Birch and Swinnerton-Dyer conjectures, the Kolyvagin Theorem and the Gross-Stark conjecture. 
The BSD conjectures relate arithmetical problems about elliptic curves to analytic problems about the associated $L$-function. 
No proof of it is shown as of today ( that is the million-dollar question!). These conjectures have been studied by many people and proved only in special cases, although there is extensive numerical evidence for their truth. The earliest results on the BSD conjectures are the Coates and Wiles result~\cite{J.A} for elliptic curves with complex multiplication, and the Gross-Zagier Theorem~\cite{G.Z}. Related results are given in \cite{K}, \cite{R} and \cite{B.C.D.T}. The very recent paper in \cite{B.S} shows that a positive proportion of elliptic curves have analytic rank 0; i.e., a positive proportion of elliptic curves have a non-vanishing $L$-function at $s=1$. Applying Kolyvagin's Theorem, it follows that a positive proportion of all elliptic curves satisfy the conjectures BSD. See also ~\cite{B.G.Z}, \cite{B.S.D}, \cite{Ca} and \cite{Cr} for numerical evidence of these conjectures. A beautiful asymptotic formula of the fourth power moment in the $q$-aspect has been obtained by Health-Brown~\cite{H81}, for $q$ prime. Recently, Bui and Heath-Brown~\cite{H2010} gave an asymptotic formula for the fourth power mean of Dirichlet $L$-functions averaged over primitive characters to modulus $q$ and over $t \in [0,T]$ which is particularly effective when $q\geq T$ ( see also~\cite{Sound} and~\cite{Y} ). \\

Concerning the value of the derivative, Stark presented a conjecture for abelian $L$-functions with simple zeros at $s=0$, expressing the value of the derivative at $s=0$ in terms of logarithms of global units ( see a series of papers~\cite{St1}, \cite{St2}, \cite{St3} and \cite{St4}). Extending the conjectures of Stark, Gross stated in 1988 a relation between the derivative of the $p$-adic $L$-function associated to $\chi$ at its exceptional zero and the $p$-adic logarithm of a $p$-unit in the extension of a totally real field $E$ cut out by $\chi$, where $\chi$ is an abelian totally odd character of $E$. In 1996 Rubin~\cite{R2} gave an extension of the conjectures of Stark, attempting to understand the values $L^{(r)}( \chi, 0)$ when the order of vanishing $r$ may be greater than one. In~\cite{D.D.P}, the authors proved the conjecture of Gross when $E$ is a real quadratic field and $\chi$ is a narrow ring class character. Recently, Ventullo~\cite{V} proposed an unconditional proof of the Gross-Stark conjecture in rank one.\\

Less is known about $L^\prime/L$ evaluated 
also at  the point $s=1$, through these values are known to be fundamental in studying the distribution of primes since Dirichlet in 1837. \\

In this paper, we show that the values $\left|L^\prime/L(1, \chi)\right|$ behave according to a distribution law. Let us state this result formally.
\begin{thm}
\label{Thm2}
There exists a unique probability measure $\mu$ such that every continuous function $f$, we have
\begin{equation}
\label{eq02}
\frac{1}{q-2}\sum_{\substack{\chi \mkern3mu \mathrel{\textsl{mod}} \mkern3mu q\\ \chi \neq \chi_0}} f\left(\left|\frac{L^\prime}{L}(1, \chi)\right|\right) \mathrel{\mathop{\longrightarrow}\limits_{q \rightarrow +\infty}}
\int\limits_{0}^{+\infty}f(t)\, d\mu (t),
\end{equation}
where $\sum_{\chi \mkern3mu \mathrel{\textsl{mod}} \mkern3mu q}$  denotes the summation over all the primitive characters $\chi \mod q$. Here the variable $q$ ranges the odd primes.
\end{thm}
We deduce the existence of $\mu$ by the general solution to the Stieltjes moment problem and the unicity by the criterion of Carleman.
  
This is an existence (and unicity) result, but getting an actual description of $\mu$ is a tantalizing problem. As we
noted before, it is likely to have a geometrical or arithmetical interpretation, on which our approach gives no information. Here is a plot of the distribution function 
\begin{equation}
\label{eq20}
D_q(t)= \frac{1}{q-2} \ \# \left\{ \chi\neq \chi_0 \mkern3mu \mathrel{\textsl{mod}} \mkern3mu q \ ; \ \left|\frac{L^\prime}{L}(1, \chi)\right| \leq t\right\},
\end{equation}
for $q=59, 101$ and $257$.
\begin{figure}[h]
\centering
\includegraphics{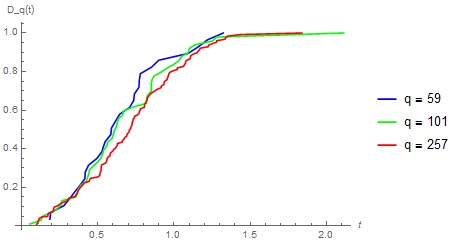}
\caption{The distribution  function $D_q(t)$.}
\label{fig}
\end{figure}
\\

The key to this result is to give an asymptotic formula of the $2k$-th power mean 
\begin{equation}
\label{eq0}
\sum_{\substack{\chi \mkern3mu \mathrel{\textsl{mod}} \mkern3mu q\\ \chi \neq \chi_0}}\left|\frac{L^\prime}{L}(1, \chi)\right|^{2k}
\end{equation}
Under the generalized Riemann hypothesis ( and later in~\cite{I.M2014} unconditionally), Ihara and Matsumoto~\cite{I.M2011} gave a stronger result related to the value-distributions of $\{ L^\prime/L(s, \chi)\}_{\chi}$ and of $\{\zeta^{\prime}/\zeta (s+i\tau)\}_{\tau}$, where $\chi$ runs over Dirichlet characters with prime conductors and $\tau$ runs over $\mathbb{R}$.

A similar study on $L(1, \chi)$ has been partially achieved; when $k=1$ by Walum~\cite{W}. The former based on the Fourier series to evaluate $\sum|L(1, \chi)|^2$ for $\chi$ ranges the odd characters modulo a prime number. In 1989, Zhang~\cite{Zh} obtained an exact formula for general integer $q \geq 3$. More recently, Louboutin~\cite{Lou} gave an exact formula for the twisted moments. His result generalizes previous works. For general $k$, Zhang and Weiqiong~\cite{Zh.W} gave an exact calculating formula for the $2k$-th power mean of $L$-functions with $k\geq 3$.\\

Here is our result.
\begin{thm}
\label{Thm1}
Let $ \epsilon >0$ and let $\chi$ be a primitive Dirichlet character modulo prime $q$. For $k$ is an arbitrary non-negative integer, we have
\[
\label{eq01}
\sum_{\substack{\chi \mkern3mu \mathrel{\textsl{mod}} \mkern3mu q\\ \chi \neq \chi_0}}\left|\frac{L^\prime}{L}(1, \chi)\right|^{2k}
= 
(q-2)\sum_{m\geq 1}\frac{\left(\sum\limits_{m=m_1\cdot m_2\cdots m_k}\Lambda(m_1)\cdots\Lambda(m_k)\right)^2}{m^2}+\Ocal_{\epsilon}\left(q^{9/10+\epsilon}\right),
\]
where $\Lambda$ is the Von Mangoldt's function.
\end{thm}
The error term is not effective as we use the full force of Siegel's theorem on exceptional zeros. 

We use an analytical method that is already used in~\cite{P.R2015} in contrast with the previous works that used essentially only elementary and combinatorial arguments. We introduce many simplifications in their intricate combinatorial argument, which is why we can handle the case of general $k$. 

This result is strong in two aspects: it is valid for general $k$ and we save a power of $q$ ( we did not try to
optimize this saving). The method of proof relies in particular on a suitable average density estimate for the zeros of Dirichlet $L$-functions. Let us note that this approach does not work for $\frac{L^\prime}{L}(1+it, \chi)$ when $t\neq 0$.\\

 We also computed 
\begin{equation*}
\sum\limits_{m\geq 1}\frac{\left(\sum\limits_{m=m_1\cdot m_2}\Lambda(m_1)\Lambda(m_2)\right)^2}{m^2}= 0.80508\cdots, 
\end{equation*}
and 
\begin{equation*}
\sum_{m\geq 1}\frac{\left(\sum\limits_{m=m_1\cdots m_4}\Lambda(m_1)\cdots\Lambda(m_4)\right)^2}{m^2}= 1.242\cdots. 
\end{equation*}
Here are some consequences of our main Theorem.  
\begin{cor}
\label{cor1}
There exists $a>0$ and  $c>0$ such that 
\begin{equation*}
\liminf_q \frac{1}{q-2} \# 
\left\{ \chi\neq \chi_0 \mkern3mu \mathrel{\textsl{mod}} \mkern3mu q \ ; \ \left|\frac{L^\prime}{L}(1, \chi)\right| \leq a\right\}\leq 1-c.
\end{equation*}
\end{cor}
To do so, it is enough to note that $\mu\neq \delta_0$ where $\delta_0$ is the measure of Dirac at $0$ ( compare the moments). Thus, there exists a compact interval $I=[a, b]$ such that $\mu(I)$ is positive. It follows that 
\begin{equation*}
\mu\left([a, b]\right)=\liminf_q \frac{1}{q-2} \# 
\left\{ \chi\neq \chi_0 \mkern3mu \mathrel{\textsl{mod}} \mkern3mu q \ ; \ \left|\frac{L^\prime}{L}(1, \chi)\right| \in [a, b]\right\}>0.
\end{equation*}     
This completes the proof. \\

Since we can control all the moments, we have an estimation for the size of the tail of our distribution. 
\begin{cor}
\label{cor2}
For $t \geq 1$. We have
\begin{equation*}
\liminf_q\frac{1}{q-2}\# \left\{ \chi\neq \chi_0 \mkern3mu \mathrel{\textsl{mod}} \mkern3mu q \ ; \ \left|\frac{L^\prime}{L}(1, \chi)\right| \geq  t\right\}
\ll
e^{-\sqrt{t}/2}.
\end{equation*}
\end{cor} 
We do not know whether this bound can be improved upon. 
\section{Notation}
We use the following notation. 
\begin{itemize}
\item[$\bullet$] Usually $s= \sigma+it$, but for a zero of $L$-function we use $\rho$.
\item[$\bullet$] If $T\geq 1$ is a real number, $N(T, \chi, \sigma)$ denotes the number of zeros $\rho$ of the function $L(s, \chi)$ in the region $|\Im \rho|\leq T$, $\Re\rho \geq \sigma$. We define, for a primitive character $\chi$ modulo $q$ : 
\[
N(T, \chi, \sigma)=\# \left\{\rho = \beta+i\gamma ; \quad |\gamma |\leq T \ and \ \beta \geq \sigma, \, L(\rho, \chi)=0 \right\}
\]
\item[$\bullet$] The notation $\underline{m}$ and  $\underline{n}$ mean that $\prod\limits_{i=1}^{k}m_i$ and $\prod\limits_{i=1}^{k}n_i$ respectively. 
\item[$\bullet$] Let $q(\underline{m} \cdot\underline{n})$ be the largest divisor of $q$ that is prime to $\underline{m}\cdot\underline{n}$. 
\end{itemize}
\section{Auxiliary lemmas}
\begin{lem}
Let $m$ and $n$ be two positive integers. We have 
\label{lem3}
\begin{equation*}
\sum\limits_{f|q}\sum_{\chi \mkern3mu \mathrel{\textsl{mod}^\star} \mkern3mu f }\chi(m)\overline{\chi(n)}=
\begin{cases}
\displaystyle{\varphi(q(mn))} \quad  &\textrm{when $m\equiv n\left[ q(mn)\right] $}\\
  \displaystyle{0} \quad  &\textrm{otherwise,}
\end{cases}
\end{equation*}
where $\chi$ ranges the primitive characters $\mkern3mu \mathrel{\textsl{mod}}\mkern3mu  f$ and $\varphi$ is the Euler's function.  
\end{lem}
\begin{proof}
See~\cite{R.R}, section~4.3.
\end{proof}
\begin{lem}
\label{toupie}
Let $M$ be an upper bound for the holomorphic function $F$ in $|s-s_0|\le R$. Assume we know of a lower bound $m>0$, for $|F(s_0)|$. Then
\begin{equation*}
\frac{F'(s)}{F(s)}=
\sum_{|\rho-s_0|\le R/2}\frac{1}{s-\rho}+
\Ocal\left(\frac{\Log (M/m)}{R}\right)
\end{equation*}
for every $s$ such that $|s-s_0|\le R/4$ and where the summation variable $\rho$ ranges the zeros $\rho$ of $F$ in the region $|\rho-s_0|\le R/2$, repeated according to multiplicity.
\end{lem}
\begin{proof}
See\cite{B1}, Lemma~3.2 and~\cite{B2}. See also~\cite{Ti}, section~3.9 [Lemma $a$].
\end{proof}
\begin{lem}
\label{ZFR}
There is a constant $c$ such that, for any non-principal character $\chi$ modulo $q$, we have
\begin{equation*}
\frac{L'}{L}(s,\chi)\ll 
\Log \left(q|t|\right)
\end{equation*}
provided that 
\begin{equation*}
    \Re s\ge 1-\frac{c}{\Log q},\quad |t|\le q
  \end{equation*}
  except for at most one of them, which we call \emph{exceptional}, and for which we
  have 
  $$\frac{L'}{L}(s,\chi)\ll_\varepsilon q^\varepsilon$$
   in the above region.
\end{lem}
\begin{lem}
  \label{density}
We have, when $\sigma\ge 4/5$ and for any $\varepsilon>0$,
  \begin{equation*}
    \sum_{f|q}\sum_{\mkern3mu \mathrel{\chi \textsl{mod}^\star} \mkern3mu f}N(T,\chi,\sigma)
    \ll_\varepsilon (qT)^{2(1-\sigma)+\varepsilon},
  \end{equation*}
  where $\chi$ ranges the primitive characters $\mkern3mu \mathrel{\textsl{mod}}\mkern3mu  f$.
\end{lem}
\begin{proof}
A proof of this Lemma can be found in~\cite{H.J}.
\end{proof}
Our main theorems will follow from the following two lemmas under the hypothesis $q$ is prime, otherwise, the combination will be complicated .
\begin{lem}
\label{lem4}
Let $m_i$, $n_i$ and $k$ be the positive integers for $i \in \{1, 2, \cdots, k\}$ and let $q$ be a prime number. Then, we have
\begin{multline*}
\sum_{\substack{m_1, \cdots,m_k, n_1, \cdots n_k\geq 1 \\ \underline{m} \equiv \underline{n} \mkern3mu \mathrel{\textsl{mod}}\mkern3mu\left[ {q\left(\underline{m} \cdot\underline{n}\right)}\right]}}\varphi\left(q(\underline{m} \cdot\underline{n})\right) \frac{\prod\limits_{i=1}^{k}\Lambda(m_i) \Lambda(n_i)}{\underline{m} \cdot\underline{n}} e^{-\underline{m} \cdot\underline{n}/X}
=\\
 (q-2)\sum_{m\geq 1}  \frac{\left(\sum\limits_{m=m_1\cdot m_2 \cdots m_k}\Lambda(m_1)\cdots \Lambda(m_k)\right)^2}{m^2}
 + 
 \Ocal_{\varepsilon}\left(qe^{-q/X}+\frac{q}{X^{1-\varepsilon}}
 \right),
\end{multline*}
where $\varphi$ is the Euler's function.  
\end{lem}
\begin{proof}
We study this summation 
\begin{equation}
\label{eq15}
F_q(X)
=
\sum_{\substack{m_1, \cdots,m_k, n_1, \cdots n_k\geq 1 \\ \underline{m} \equiv \underline{n} \mkern3mu \mathrel{\textsl{mod}}\mkern3mu \left[ {q\left(\underline{m} \cdot\underline{n}\right)}\right]}}\varphi\left(q(\underline{m} \cdot\underline{n})\right) \frac{\prod\limits_{i=1}^{k}\Lambda(m_i) \Lambda(n_i)}{\underline{m} \cdot\underline{n}} e^{-\underline{m} \cdot\underline{n}/X}. 
\end{equation}
We split the domain defined by the condition $\underline{m} \equiv \underline{n} \mkern3mu \mathrel{\textsl{mod}}\mkern3mu \left[ {q(\underline{m} \cdot\underline{n})}\right]$ in fourth parts.
\begin{itemize}
\item[$\bullet$] The first case is when $(q,\underline{m} \cdot\underline{n})= 1$ and $\underline{m} \neq \underline{n}$. It follows that $q(\underline{m} \cdot\underline{n})=q$, and thus $\underline{m} \equiv \underline{n}\mkern3mu \mathrel{\textsl{mod}}\mkern3mu [q]$. We get the contribution
\begin{eqnarray*}
A_q(X) &=& \sum_{\substack{m_1, \cdots,m_k, n_1, \cdots n_k\geq 1 \\ \underline{m} \equiv \underline{n} \mkern3mu \mathrel{\textsl{mod}}\mkern3mu [q]\\ \underline{m} \neq \underline{n} }}\varphi(q(\underline{m} \cdot\underline{n})) \prod\limits_{i=1}^{k}\Lambda(m_i) \Lambda(n_i) \frac{e^{-\underline{m} \cdot\underline{n}/X}}{\underline{m} \cdot\underline{n}}
\\&= & 
 \varphi (q) \sum_{\substack{\underline{m}, \underline{n} \geq 1\\ \underline{m} \equiv \underline{n} \mkern3mu \mathrel{\textsl{mod}}\mkern3mu [q]\\ \underline{m} \neq \underline{n}}} \left(\sum_{\underline{m}=\prod\limits_{i=1}^{k} m_i} \Lambda(m_1)\cdots \Lambda(m_k) \sum_{\underline{n}=\prod\limits_{i=1}^{k} n_i}\Lambda(n_1)\cdots \Lambda(n_k)\right)  \frac{e^{-\underline{m}\cdot\underline{n}/X}}{\underline{m}\cdot\underline{n}}. 
 \end{eqnarray*}
 By using the fact that 
 \begin{equation*}
 \sum_{\underline{m}=\prod\limits_{i=1}^{k} m_i} \Lambda(m_1)\cdots \Lambda(m_k) \sum_{\underline{n}=\prod\limits_{i=1}^{k} n_i}\Lambda(n_1)\cdots \Lambda(n_k)\ll_{\varepsilon} \underline{m}^{\varepsilon}\underline{n}^{\varepsilon}, 
 \end{equation*}
 we get
 \begin{eqnarray*}
 A_q(X) 
& \ll_{\varepsilon} & 
\varphi (q) \sum_{\substack{ \underline{m}, \underline{n} \geq 1\\ \underline{m} \equiv \underline{n} \mkern3mu \mathrel{\textsl{mod}}\mkern3mu [q]\\ \underline{m} < \underline{n}}} \frac{e^{-\underline{m}\cdot\underline{n}/X}}{\underline{m}^{1-\varepsilon}\underline{n}^{1-\varepsilon}}
\\& \ll_{\varepsilon} & 
\varphi (q) \sum_{\underline{m}\geq 1}\sum_{\substack{ k\geq 1\\ \underline{n}=\underline{m}+kq}} \frac{e^{-\underline{m}\cdot\underline{n}/X}}{\underline{m}^{1-\varepsilon}\underline{n}^{1-\varepsilon}}
=
\varphi (q) \sum_{\underline{m}\geq 1}\sum_{ k\geq 1} \frac{e^{-\underline{m}(\underline{m}+kq)/X}}{\underline{m}^{1-\varepsilon}\left(\underline{m}+kq\right)^{1-\varepsilon}}
\\& \ll_{\varepsilon} &  
\varphi (q)\sum_{ k\geq 1} e^{-kq/X}
 \ll_{\varepsilon}  
\varphi (q)e^{-q/X}.
\end{eqnarray*}
Therefore, we have 
\begin{equation}
\label{eq1}
\boxed{A_q(X)
\ll_\varepsilon
q e^{-q/X}.}
\end{equation}
\item[$\bullet$] The second case is when $(q,\underline{m} \cdot\underline{n})= 1$ and $\underline{m} = \underline{n}$, it follows that $q(\underline{m})=q$. We get the contribution
\begin{equation*}
B_q(X) = \sum_{\substack{m_1, \cdots,m_k, n_1, \cdots n_k\geq 1 \\ (q,\underline{m}\cdot \underline{n})=1, \ \underline{m}= \underline{n}}} \varphi(q(\underline{m} \cdot\underline{n})) \prod\limits_{i=1}^{k}\Lambda(m_i) \Lambda(n_i) \frac{e^{-\underline{m} \cdot\underline{n}/X}}{\underline{m} \cdot\underline{n}}.
\end{equation*}
Let us write the function $B_q(X)$ as follows 
\begin{equation*}
B_q(X)= B^{\sharp}_q(X)+B^{\flat}_q(X), 
\end{equation*}
where 
\begin{equation*}
B^{\sharp}_q(X)= \sum_{\substack{m_1, \cdots,m_k, n_1, \cdots n_k\geq 1 \\ (q,\underline{m}\cdot \underline{n})=1, \ \underline{m}= \underline{n} \\ \underline{m} \leq X^{1/2} }} \varphi\left(q( \underline{m})\right)\prod\limits_{i=1}^{k}\Lambda(m_i) \Lambda(n_i) \frac{e^{-{\underline{m}}^2/X}}{\underline{m}^2}, 
\end{equation*}
and 
\begin{equation*}
B^{\flat}_q(X)= \sum_{\substack{m_1, \cdots,m_k, n_1, \cdots n_k\geq 1 \\ (q,\underline{m}\cdot \underline{n})=1, \ \underline{m}= \underline{n}\\ \underline{m} > X^{1/2} }} \varphi\left(q( \underline{m})\right)\prod\limits_{i=1}^{k}\Lambda(m_i) \Lambda(n_i) \frac{e^{-{\underline{m}}^2/X}}{\underline{m}^2}. 
\end{equation*}
\begin{itemize}
\item[(A)] \textbf{For the function $B^{\flat}_q(X)$.} Since $\underline{m}>X^{1/2}$, it follows that  $e^{-{\underline{m}}^2/X} \leq 1$. Thus, we get 
\begin{equation*}
B^{\flat}_q(X)
\ll 
\varphi(q)\sum_{\substack{m_1, \cdots,m_k, n_1, \cdots n_k\geq 1 \\
 (q,\underline{m}\cdot \underline{n})=1, \ \underline{m}= \underline{n}\\ \underline{m} > X^{1/2} }}  \frac{\prod\limits_{i=1}^{k}\Lambda(m_i) \Lambda(n_i)}{\underline{m}^2}. 
\end{equation*}
Using again the fact that $\prod\limits_{i=1}^{k}\Lambda(m_i) \Lambda(n_i)\ll_{\varepsilon} \underline{m}^{\varepsilon}\cdot \underline{n}^{\varepsilon} \ll_{\varepsilon} \underline{m}^{2\varepsilon}$. We have
\begin{equation}
\label{eq2}
B^{\flat}_q(X)
\ll_{\varepsilon} 
\frac{q}{X^{1-\varepsilon}}.
\end{equation}
\item[(B)] \textbf{For the function $B^{\sharp}_q(X)$.} Since $ \underline{m}^2$ is small enough, we can rely on the approximation 
\begin{equation*}
e^{-\underline{m}^2/X}=1+\Ocal{\left(\frac{\underline{m}^2}{X}\right)},
\end{equation*}
which gives us
\begin{multline*}
B^{\sharp}_q(X)
=
\varphi(q)\sum_{\substack{m_1, \cdots,m_k, n_1, \cdots n_k\geq 1 \\(q,\underline{m}\cdot \underline{n})=1,\ \underline{m} =\underline{n} \\ \underline{m} \leq X^{1/2} }}   \frac{\prod\limits_{i=1}^{k}\Lambda(m_i) \Lambda(n_i)}{\underline{m}^2}
\\+ \Ocal\left(\frac{\varphi(q)}{X}\sum_{\substack{m_1, \cdots,m_k, n_1, \cdots n_k\geq 1 \\ (q,\underline{m}\cdot \underline{n})=1,\ \underline{m}= \underline{n} \\ \underline{m} \leq X^{1/2} }}  \prod\limits_{i=1}^{k}\Lambda(m_i) \Lambda(n_i)\right). 
\end{multline*}
By a similar argument as above, we get
\begin{equation*}
B^{\sharp}_q(X)
=
\varphi(q)\sum_{\substack{m_1, \cdots,m_k, n_1, \cdots n_k\geq 1 \\ (q,\underline{m}\cdot \underline{n})=1,\ 
 \underline{m}= \underline{n} \\  \underline{m} \leq X^{1/2} }}   \frac{\prod\limits_{i=1}^{k}\Lambda(m_i) \Lambda(n_i)}{\underline{m}^2}+ \Ocal_{\varepsilon}\left(\frac{\varphi(q)}{X^{1-\varepsilon}}\right).
\end{equation*}
Thus, we have
 \begin{equation}
 \label{eq3}
 B^{\sharp}_q(X)
 =
 (q-2)\sum_{\substack{m_1, \cdots,m_k, n_1, \cdots n_k\geq 1 \\ (q,\underline{m}\cdot \underline{n})=1,\ 
  \underline{m}= \underline{n}}}  \frac{\prod\limits_{i=1}^{k}\Lambda(m_i) \Lambda(n_i)}{\underline{m}^2}+ \Ocal_{\varepsilon}\left(\frac{q}{X^{1-\varepsilon}} \right).
 \end{equation}
 \end{itemize}
From Eq~\eqref{eq2} and~\eqref{eq3}, we find that 
\begin{equation}
\label{eq4}
\boxed{B_q(X)= (q-2)\sum_{\substack{m_1, \cdots,m_k, n_1, \cdots n_k\geq 1 \\ (q,\underline{m}\cdot \underline{n})=1,\ \underline{m}= \underline{n} }} \frac{\prod\limits_{i=1}^{k}\Lambda(m_i) \Lambda(n_i)}{\underline{m}^2}+ \Ocal_{\varepsilon}\left(\frac{q}{X^{1-\varepsilon}}\right).}
\end{equation}
\item[$\bullet$] The third case is when $(q,\underline{m} \cdot\underline{n})\neq 1$ and $\underline{m} \neq \underline{n}$. It follows that $q(\underline{m} \cdot\underline{n})=1$ and thus $\varphi\left(q(\underline{m} \cdot\underline{n})\right)=1$. We get the contribution 
\begin{eqnarray*}
C_q(X) 
&=&
\sum_{\substack{m_1, \cdots,m_k, n_1, \cdots n_k\geq 1 \\ q|\underline{m} \cdot\underline{n}, \ \underline{m} \neq \underline{n} }}\varphi\left(q(\underline{m} \cdot\underline{n})\right) \frac{\prod\limits_{i=1}^{k}\Lambda(m_i) \Lambda(n_i)}{\underline{m} \cdot\underline{n}} \ e^{-\underline{m} \cdot\underline{n}/X}
\\&= & 
 \sum_{\substack{\ell \geq 1 \\ q|\ell}} \sum_{\ell = \underline{m} \cdot\underline{n}} \prod\limits_{i=1}^{k}\Lambda(m_i) \Lambda(n_i)  \frac{e^{-\ell/X}}{\ell }. 
\end{eqnarray*}
We notice that 
\begin{equation*}
\sum_{\ell = \underline{m} \cdot\underline{n}}\prod\limits_{i=1}^{k}\Lambda(m_i) \Lambda(n_i) \ll_{\varepsilon} \ell^{\varepsilon}.
\end{equation*}
Thus, we have
\begin{equation*}
C_q(X) 
\ll_{\varepsilon} 
\sum_{\substack{\ell \geq q \\ q|\ell}}\frac{e^{-\ell/X}}{\ell^{1-\varepsilon} }.
\end{equation*}
Since $q|\ell$, we write $\ell =qu$. Then
\begin{equation}
\label{eq5}
\boxed{C_q(X) 
\ll_{\varepsilon} 
\frac{1}{q^{1-\varepsilon}}\sum_{u \geq 1}\frac{e^{-qu/X}}{u^{1-\varepsilon} }
\ll_{\varepsilon} 
\frac{e^{-q/X}}{q^{1-\varepsilon}}.} 
\end{equation}
\item[$\bullet$] The fourth case is when $(q,\underline{m} \cdot\underline{n})\neq 1$ and $\underline{m} = \underline{n}$. It follows that $q(\underline{m})=1$. We get the contribution 
\begin{eqnarray*}
D_q(X) 
&=&
\sum_{\substack{m_1, \cdots,m_k, n_1, \cdots n_k\geq 1 \\ q| \underline{m}, \ \underline{m} = \underline{n} }}\varphi\left(q(\underline{m} \cdot\underline{n})\right) \prod\limits_{i=1}^{k}\Lambda(m_i) \Lambda(n_i) \frac{ e^{-\underline{m} \cdot\underline{n}/X}}{\underline{m} \cdot\underline{n}}
\\&= & 
 \sum_{\substack{m_1, \cdots,m_k, n_1, \cdots n_k \geq 1 \\ q|\underline{m}, \ \underline{m} = \underline{n}}} \prod\limits_{i=1}^{k}\Lambda(m_i) \Lambda(n_i)  \frac{e^{-\underline{m}^2/X}}{\underline{m}^2}. 
\end{eqnarray*}
It follows that
\begin{eqnarray*}
D_q(X)
&=&\sum_{\substack{\underline{m}, \underline{n} \geq 1 \\ q| \underline{m}, \ \underline{m} = \underline{n} }}\left(\sum_{\underline{m}=\prod\limits_{i=1}^{k}m_i}\Lambda(m_1)\cdots\Lambda(m_k)\sum_{\underline{n}=\prod\limits_{i=1}^{k}n_i}\Lambda(n_1)\cdots \Lambda(n_k)\right)\frac{ e^{-\underline{m}^2/X}}{\underline{m}^2}
\\&\ll_{\varepsilon}&
\sum_{\substack{\underline{m} \geq 1 \\ q| \underline{m} }} \frac{ e^{-\underline{m}^2/X}}{\underline{m}^{2(1-\varepsilon)}}= \sum\limits_{u\geq 1}\frac{e^{-(qu)^2/X}}{(qu)^{2(1-\varepsilon)}}
\ll_{\varepsilon}
\frac{e^{-q^2/X}}{q^{2(1-\varepsilon)}}.
\end{eqnarray*}
Therefore, we have 
\begin{equation}
\label{eq6}
\boxed{D_q(X)
\ll_{\varepsilon}
\frac{e^{-q^2/X}}{q^{2(1-\varepsilon)}}.}
\end{equation}
From Eq~\eqref{eq1}, \eqref{eq4}, \eqref{eq5} and \eqref{eq6}, we conclude that
\begin{equation*}
F_q(X)= (q-2)\sum_{\substack{m_1, \cdots,m_k, n_1, \cdots n_k\geq 1 \\ \left(q,  \underline{m}\cdot\underline{n}\right)=1, \
 \underline{m}= \underline{n} }}  \frac{\prod\limits_{i=1}^{k}\Lambda(m_i) \Lambda(n_i)}{\underline{m}^2}
 + \Ocal_{\varepsilon}\left(qe^{-q/X}+\frac{q}{X^{1-\varepsilon}}\right).
\end{equation*}
\end{itemize}
This completes the proof.
\end{proof}
\begin{lem}
\label{lem5}
For any integer number $k\geq 1$, we have
\begin{equation}
\label{eq11}
\sum_{m_1\cdot m_2 \cdots m_k=m}\Lambda(m_1)\cdots \Lambda(m_k)\leq (\log m)^k
\end{equation}
\end{lem}
\begin{proof}
We prove this lemma by induction on $k$. For $k=1$ is easily. In order to show that Eq~\eqref{eq11} is valid for $k=2$, we write
\begin{equation*}
\sum_{m_1 m_2 = m}\Lambda(m_1)\Lambda(m_2)
\le \log m \sum_{m_1 m_2 = m}\Lambda(m_2) \leq (\log m)^2
\end{equation*}
Now, we assume that Eq~\eqref{eq11} is valid for any fixed and non-negative integer $k-1$. Then we have to prove that it is also valid for $k$. By induction hypothesis, we have   
\begin{eqnarray*}
\sum_{m_1\cdot m_2 \cdots m_k =m}\Lambda(m_1)\cdots \Lambda(m_k)
&=& 
\sum_{m_1n=m}\Lambda(m_1)\sum_{m_2\cdot m_3 \cdots m_k=n}\Lambda(m_2)\cdots \Lambda(m_k)
\\&\leq &
\sum_{m_1n=m}\Lambda(m_1) \log^{k-1}n \leq (\log m)^k
\end{eqnarray*}
We conclude from the above that Eq~\eqref{eq11} is valid for $k$. Then it is valid for all $k\geq 1$. The lemma is proved.
\end{proof}
\section{Proof of Theorem~\ref{Thm1}}
For $q$ is prime, we consider the function 
\begin{equation*}
G_q(s)=\sum_{\substack{\chi \mkern3mu \mathrel{\textsl{mod}} \mkern3mu q\\ \chi \neq \chi_0}}\left(\frac{L^{\prime}}{L}(s, \chi)\right)^k\left(\frac{L^\prime}{L}(s, \bar{\chi})\right)^k
\end{equation*}
where $\chi$ ranges the non-principle primitive characters modulo $q$. When $\Re s>1$, the series converges absolutely. Applying the following definition of $L^\prime/L$ 
\begin{equation*}
\frac{L^\prime}{L}(s, \chi)=\sum_{n\geq 1}\frac{\chi(n)\Lambda(n)}{n^s},
\end{equation*}
we write the function $G_q(s)$ as 
\begin{equation*}
G_q(s)=
\sum_{\substack{\chi \mkern3mu \mathrel{\textsl{mod}} \mkern3mu q\\ \chi \neq \chi_0}}\sum_{\substack{m_1\cdots m_k \geq 1 \\n_1\cdots n_k \geq 1
 }}\frac{\prod\limits_{i=1}^{k} \Lambda(m_i)\chi(m_i) \prod\limits_{i=1}^{k} \Lambda(n_i)\bar{\chi}(n_i)}{\left(\prod\limits_{i=1}^{k} m_in_i\right)^s}.
\end{equation*}
The proof relies on two distinct evaluations of the quantity:
\begin{equation}
\label{eq16}
S_q(X)=\frac{1}{2}\int\limits_{2-i\infty}^{2+i\infty}G_q(s)X^{s-1}\Gamma(s-1)\, ds. 
\end{equation}
\begin{enumerate}
\item
\textbf{The first evaluation} relies on the formula $e^{-y}=\frac{1}{2i\pi} \int_{2-i\infty}^{2+i\infty}y^{-s}\Gamma(s)\, ds$ (valid for positive $y$) and on using Lemma~\ref{lem3}. We readily find that
\begin{equation*}
S_q(X)=\sum_{\substack{ m_1,\cdots, m_k, n_1, \cdots, n_k\geq 1 \\ \underline{m} \equiv \underline{n} \mkern3mu \mathrel{\textsl{mod}}\mkern3mu\left[ {q(\underline{m} \cdot\underline{n})}\right] }}\varphi\left(q(\underline{m} \cdot\underline{n})\right)\frac{\prod\limits_{i=1}^{k}\Lambda(m_i)\Lambda(n_i)}{\underline{m} \cdot\underline{n}}e^{-\underline{m} \cdot\underline{n}/X}.
\end{equation*}
Thanks to Lemma~\ref{lem4}, we get 
\begin{equation}
\label{eq25}
\boxed{S_q(X)
=
(q-2)\sum_{m\geq 1}  \frac{\left(\sum\limits_{m=m_1\cdots m_k}\Lambda(m_1)\cdots \Lambda(m_k)\right)^2}{m^2}
 + \Ocal_{\varepsilon}\left(qe^{-q/X}+\frac{q}{X^{1-\varepsilon}}\right).}
\end{equation}
\item
\textbf{The second evaluation:} On selecting $\sigma=9/10$, $\varepsilon=1/10$ and $T=q$ in
Lemma~\ref{density}, we see that at 
most $\Ocal(q^{3/5})$ characters modulo a divisor of $q$ have a zero in the region
\begin{equation*}
|\Im \rho|\le q,\ \Re \rho\geq 9/10.
\end{equation*}
We call the characters which are in this set \textbf{bad characters} and the other set, the one of \textbf{good characters}. Now, we shift the line of integration in Eq \eqref{eq16} to 
\begin{description}
\item[$\bullet$] 
 $\Re s=9/10$ and $|\Im s|\le q$ when $\chi$ belongs to the good set;
\item[$\bullet$] $\Re s=1-c/\log q$ and $|\Im s|\le q$ when $\chi$ belongs to the bad set; Here $c$ is the constant from Lemma~\ref{ZFR}. 
\end{description}

Then, we have: 
\begin{itemize}
\item[(i)] 
For the bad character, when $\Re s=1-c/\log q$ and $|\Im s|\le q$. Lemma \ref{ZFR} gives us that 
\begin{equation*}
L'/L(s,\chi)\ll \log q. 
\end{equation*}
We have 
\begin{equation*}
\left|X^{s-1}\right| \leq X^{-\frac{c}{\log q}}.
\end{equation*} 
Thus  
\begin{eqnarray*}
\int\limits_{1-\frac{c}{\log q}-iq}^{1-\frac{c}{\log q}+iq}\left|G_q(s)X^{s-1}\Gamma(s-1)\, ds\right|
&\leq& 
q^{3/5}(\log q)^{2k}X^{-\frac{c}{\log q}}\int\limits_{1-\frac{c}{\log q}-iq}^{1-\frac{c}{\log q}+iq}\left|\Gamma(s-1)\right|\, ds
\\&\leq &
q^{3/5}(\log q)^{2k}X^{-\frac{c}{\log q}}\int\limits_{-q}^{q}\left|\Gamma\left(-\frac{c}{\log q}+it\right)\right|\, dt
\end{eqnarray*}
Recall that the formula of complex Stirling is given by 
\begin{equation*}
\Gamma(z+a)=\sqrt{2\pi} \, e^{-z}z^{z+a-1/2}\left(1+\Ocal\left(1/|z|\right)\right),
\end{equation*}
where $a$ is a complex number fixed, $|\arg z| \leq \pi$
and $|z|\geq 1$. Thanks to this formula, we deduce that 
\begin{equation*}
\left|\Gamma\left(-\frac{c}{\log q}+it\right)\right| \ll \sqrt{2\pi} \, |t|^{-\frac{c}{\log q}-\frac{1}{2}}\, e^{-\pi t/2},
\end{equation*}
and that 
\begin{equation}
\label{eq7}
\int\limits_{1-\frac{c}{\log q}-iq}^{1-\frac{c}{\log q}+iq}\left|G_q(s)X^{s-1}\Gamma(s-1)\, ds\right| 
\ll 
q^{3/5}(\log q)^{2k}X^{-\frac{c}{\log q}}.
\end{equation}
\item[(ii)]
Even for the exceptional one, we have: 
\begin{equation*}
L'/L(s,\chi)\ll_{\varepsilon} q^{\varepsilon}, 
\end{equation*}
when $\Re s= 1-c/q^\varepsilon$ and $|\Im s|\leq 1$. Then, we get 
\begin{eqnarray}
\label{eq8}
\int\limits_{1-\frac{c}{ q^\varepsilon}-i}^{1-\frac{c}{ q^\varepsilon}+i}\left|G_q(s)X^{s-1}\Gamma(s-1)\, ds\right|
&\leq& 
q^{2k\varepsilon}X^{-\frac{c}{q^\varepsilon}}\int\limits_{-1}^{1}\left|\Gamma\left(-\frac{c}{ q^\varepsilon}+it\right)\right|\, dt \nonumber
\\&\ll_{\varepsilon} &
q^{2k\varepsilon}X^{-\frac{c}{ q^\varepsilon}} \log^\varepsilon q
\end{eqnarray}
\item[(iii)]
For a good character, when $\Re s=\frac{9}{10}$ and $|\Im s|\le q$. Lemma~\ref{toupie} gives us that
\begin{equation*}
L'/L(s,\chi)\ll \log q.  
\end{equation*}
 We have also 
\begin{equation*}
\left|X^{s-1}\right| \leq X^{-\frac{1}{10}}.
\end{equation*} 
We deduce that 
\begin{eqnarray*}
\int\limits_{\frac{9}{10}-iq}^{\frac{9}{10}+iq}\left|G_q(s)X^{s-1}\Gamma(s-1)\, ds\right|
&\leq& 
\varphi(q)X^{-\frac{1}{10}} \log^{2k} q \int\limits_{\frac{9}{10}-iq}^{\frac{9}{10}+iq}\left|\Gamma(s-1)\right|\, ds
\\&\leq &
q X^{-\frac{1}{10}} \log^{2k} q \int\limits_{0}^{1}\left|\Gamma\left(-\frac{1}{10}+it\right)\right|\, dt
\end{eqnarray*}
Applying again the formula of Stirling, we get 
\begin{equation*}
\left|\Gamma\left(-\frac{1}{10}+it\right)\right| \ll \sqrt{2\pi} \, |t|^{-\frac{3}{5}}\, e^{-\pi t/2}. 
\end{equation*}
Therefore, we have
\begin{equation}
\label{eq9}
\int\limits_{\frac{9}{10}-iq}^{\frac{9}{10}+iq}\left|G_q(s)X^{s-1}\Gamma(s-1)\, ds\right| 
\ll 
q X^{-\frac{1}{10}} \log^{2k} q
\end{equation}
\item[(vi)] For all characters in the region $\Re s=2$ and $\left|\Im s\right| \geq q$, we find that
\begin{equation}
\label{eq10}
\int\limits_{\substack{\Re s=2\\  \left|\Im s\right| \geq q}}\left|G_q(s)X^{s-1}\Gamma(s-1)\, ds\right| 
\ll 
 X q \log ^{2k}q\int\limits_{t\geq q}\left|\Gamma(1+it)\right|\, dt 
\ll 
X q \log ^{2k}q \, e^{-q\pi/2}
\end{equation}
Since $q$ is a prime number. From Eq \eqref{eq7}, \eqref{eq8}, \eqref{eq9} and \eqref{eq10}, we obtain that 
\begin{equation}
\label{eq18}
\boxed{S_q(X)=\sum_{\substack{\chi \mkern3mu \mathrel{\textsl{mod}} \mkern3mu q\\ \chi \neq \chi_0}}
    \left|\frac{L'}{L}(1,\chi)\right|^{2k}
    +\Ocal\left(
    q^{3/5}X^{-\frac{c}{\log q}}+qX^{-\frac{1}{10}}\right).}
\end{equation}
The first term on the right-hand side above comes from case $s=1$. 
\end{itemize}
 From  Eq \eqref{eq25} and \eqref{eq18}, we conclude that 
 \begin{multline*}
\sum_{\substack{\chi \mkern3mu \mathrel{\textsl{mod}} \mkern3mu q\\ \chi \neq \chi_0}}
\left|\frac{L'}{L}(1,\chi)\right|^{2k}
 =
(q-2)\sum_{m\geq 1}  \frac{\left(\sum\limits_{m=m_1\cdots m_k}\Lambda(m_1)\cdots \Lambda(m_k)\right)^2}{{m}^2}
 \\
 +\Ocal_{\varepsilon}\left(qe^{-q/X}+\frac{q}{X^{1-\varepsilon}}+
    q^{3/5}X^{-\frac{c}{\log q}}+qX^{-\frac{1}{10}}\right)
 \end{multline*}
 For $X= q^{1-\varepsilon}$ and $\varepsilon>0$, we conclude that 
 \begin{equation*}
 \sum_{\substack{\chi \mkern3mu \mathrel{\textsl{mod}} \mkern3mu q\\ \chi \neq \chi_0}}
 \left|\frac{L'}{L}(1,\chi)\right|^{2k}
  =
(q-2)\sum_{m\geq 1}  \frac{\left(\sum\limits_{m=m_1 \cdots m_k}\Lambda(m_1)\cdots \Lambda(m_k)\right)^2}{{m}^2}+\Ocal_{\varepsilon}\left(q^{9/10+\varepsilon}\right).
 \end{equation*}
 \end{enumerate}
 This completes the proof.
 \section{Proof of Theorem~\ref{Thm2}}
 We recall that $\mu_q$ is defined by $\mu_q\left([0, t]\right)= D_q(t)$, where $D_q(t)$ is given by Eq~\eqref{eq20}. Then, we have $\mu_q$ is non-negative and $\mu_q ([0, \infty[ )=1$. 
Setting 
 \begin{equation*}
m_k= m_k(q)=\int\limits_{0}^{\infty}t^k \, d\mu_q(t),
 \end{equation*}
 and 
\begin{equation*}
\Delta_k(q)=
  \left| {\begin{array}{ccccc}
   m_0 & m_1 & m_2 & \cdots& m_k\\
   m_1 & m_2& m_3& \cdots& m_{k+1}\\
   m_2& m_3& m_4& \cdots& m_{k+2}\\
  \vdots& \vdots& \vdots& \ddots& \vdots\\
   m_k& m_{k+1}& m_{k+2}& \cdots& m_{2k}
  \end{array} } \right|,
  \quad 
  \Delta^\star_k(q)=
    \left| {\begin{array}{ccccc}
     m_1 & m_2 & m_3 & \cdots& m_{k+1}\\
     m_2 & m_3& m_4& \cdots& m_{k+2}\\
     m_3& m_4& m_5& \cdots& m_{k+3}\\
     \vdots& \vdots& \vdots& \ddots& \vdots\\
     m_{k+1}& m_{k+2}& m_{k+3}& \cdots& m_{2k+1}
    \end{array} } \right|
\end{equation*}
 For $q$ is fixed, we have $\Delta_k(q)$ and $\Delta^\star_k(q)$ are non-negative. Thus, when $q$ goes to infinity, we get  $\Delta_k(\infty)$ and $\Delta^\star_k(\infty)$ are non-negative for all $k\geq 1$. By the solution of the Stieltjes moment problem, we deduce that $\mu$ exists.
 
In order to complete our proof, it remains to show that $\mu$ unique. Define 
 \begin{equation}
 \label{eq19}
 M_k= \sum_{m\geq 1}\frac{\left(\sum\limits_{m=m_1\cdot m_2\cdots m_k}\Lambda(m_1)\cdots\Lambda(m_k)\right)^2}{m^2}.
 \end{equation}
Using Lemma~\ref{lem5}, we get 
\begin{equation*}
M_k \leq \sum_{m\geq 2^k} \frac{(\log m)^{2k}}{m^2}.
\end{equation*}
Now, we notice that 
\begin{eqnarray*}
\sum_{m\geq 2^k} \frac{(\log m)^{2k}}{m^2} 
&\leq  & 
\int_{2^k}^{\infty} \frac{(\log t)^{2k}}{t^2}\, dt + \frac{(\log e^k)^{2k}}{e^{2k}}
\\&\leq & \int_{k\log 2}^{\infty} u^{2k}\, e^{-u}\, du +
\left(\frac{k}{e}\right)^{2k}
\leq  \Gamma(2k+1) +\left(\frac{k}{e}\right)^{2k}
= (2k)!+ \left(\frac{k}{e}\right)^{2k}.
\end{eqnarray*}
Then, we have 
\begin{equation}
\label{eq21}
M_k \leq (2k)!+\left(\frac{k}{e}\right)^{2k}.
\end{equation}
Therefore, we get
\begin{equation*}
\sum_{k_\geq 1}\frac{1}{M_k^{\frac{1}{2k}}} \gg \sum_{k\geq 1}\left(\frac{1}{(2k)!}\right)^{\frac{1}{2k}}+
\sum_{k\geq 1} \frac{e}{k}
= \infty
\end{equation*}
It follows that the condition of Carleman is checked and thus the function $\mu$ is unique. This completes the proof. 
 \section{Proof of Corollary~\ref{cor2}}
We note that 
\begin{equation*}
\frac{1}{q-2}\sum_{\substack{\chi \mkern3mu \mathrel{\textsl{mod}} \mkern3mu q\\ \chi \neq \chi_0}}\left|\frac{L^\prime}{L}(1, \chi)\right|^{2k}
\geq 
\frac{t^{2k}}{q-2}\# \left\{ \chi\neq \chi_0 \mkern3mu \mathrel{\textsl{mod}} \mkern3mu q \ ; \ \left|\frac{L^\prime}{L}(1, \chi)\right| \geq  t\right\}
\end{equation*} 
Then, we have 
\begin{equation*}
\frac{1}{q-2}\# \left\{ \chi\neq \chi_0 \mkern3mu \mathrel{\textsl{mod}} \mkern3mu q \ ; \ \left|\frac{L^\prime}{L}(1, \chi)\right| \geq  t\right\}
\leq 
\frac{M_k}{t^{2k}}
\end{equation*} 
where $M_{k}$ is given by Eq~\eqref{eq19} and $M_{k} \ll (2k)!$ by the preceding proof. Thanks to the Stirling formula, we write 
\begin{equation*}
\frac{1}{q-2}\# \left\{ \chi\neq \chi_0 \mkern3mu \mathrel{\textsl{mod}} \mkern3mu q \ ; \ \left|\frac{L^\prime}{L}(1, \chi)\right| \geq  t\right\}
\ll
\left(\frac{4k^2}{e^2\, t}\right)^{k} \sqrt{4\pi k}.
\end{equation*}
Taking $k= \sqrt{t}$, we get 
\begin{equation*}
\liminf_q\frac{1}{q-2}\# \left\{ \chi\neq \chi_0 \mkern3mu \mathrel{\textsl{mod}} \mkern3mu q \ ; \ \left|\frac{L^\prime}{L}(1, \chi)\right| \geq  t\right\}
\ll
e^{-\sqrt{t}/2}.
\end{equation*}
This completes the proof.
 \section{Scripts}
We present here an easier GP script for computing the values $\left|L^\prime /L(1, \chi)\right|$. In this loop, we use the Pari package " ComputeL" written by Tim Dokchitser to compute values of $L$- functions and its derivative. This package is available on-line at 
 \begin{center}
 {\tt www.maths.bris.ac.uk/\~{}matyd/}
 \end{center}
 On this base we write the next script. I am indebted to Olivier Ramar\'e for helping me in writing it. We simply plot Figure~\eqref{fig} via \\
 
  \begin{verbatim}
   read("computeL"); /* by Tim Dokchitser */
  default(realprecision,28); 
  {run(p=37)=
     local(results, prim, avec);
     prim = znprimroot(p);
     results = vector(p-2, i, 0);
     for(b = 1, p-2,
        avec = vector(p,k,0);
        for (k = 0, p-1, avec[lift(prim^k)+1]=exp(2*b*Pi*I*k/(p-1)));
        conductor = p; 
        gammaV    = [1];
        weight    = b%2; 
        sgn       = X;
        initLdata("avec[k%p+1]",,"conj(avec[k%p+1])"); 
        sgneq = Vec(checkfeq());
        sgn   = -sgneq[2]/sgneq[1]; 
        results[b] = abs(L(1,,1)/L(1));
           \\print(results[b]);
        );
     return(results);
  }

  {goodrun(borneinf, bornesup)=
     forprime(p = borneinf, bornesup,
              print("------------------------");
              print("p = ",p);
              print(vecsort(run(p))));}
  \end{verbatim}

\mytitle{Sumaia Saad Eddin}\\
Institute of Financial Mathematics\\
and Applied Number Theory,\\ Johannes Kepler University Linz,\\ Altenbergerstrasse 69, 4040 Linz, Austria.\\
{sumaia.saad\_eddin@jku.at}


\begin{thebibliography}{00}
\bibitem{B.S} 
 M. Bhargava, A. Shankar,
 \textit{Ternary cubic forms having bounded invariants, and the existence of a positive proportion of elliptic curves having rank $0$}, 
Annals of Mathematics 181 (2015) no. 2 587-–621.
 
    \bibitem{B.S.D} 
B. Birch and H.P.F. Swinnerton-Dyer,
\textit{Notes on elliptic curves II},
J. reine angew. Math. 218 (1965) 79-–108.
 
\bibitem{B.C.D.T} 
C. Breuil, B. Conrad, F. Diamond and R. Taylor,
\textit{On the Modularity of Elliptic Curves over $\mathbb{Q}$: Wild $3$-Adic Exercises},
Journal of the American Mathematical Society 14 (2001), no. 4 843--939.

 \bibitem{B.G.Z} 
J. Buhler, B.H. Gross and D.B. Zagier,
\textit{On the conjecture of Birch and Swinnerton-Dyer
for an elliptic curve of rank $3$},
Math. Comp. 44 (1985) 473--481.
          
 \bibitem{H2010} 
 H. M. Bui and D. R. Heath-Brown,
\textit{A note on the fourth moment of Dirichlet $L$-functions},
 Acta Arith 141 (2010) 335--344.
 
 
\bibitem{Ca} 
J.W.S. Cassels,
\textit{Arithmetic on curves of genus $1$. VIII. On conjectures of Birch and Swinnerton-Dyer},
J. reine angew. Math. 217 (1965) 180–-199.
          
 \bibitem{J.A} 
 J. Coates and A. Wiles,
\textit{On the conjecture of Birch and Swinnerton-Dyer},
Inventiones Mathematicae 39 (1977) no. 3  223--251.
       
  
    \bibitem{Cr} 
J.E. Cremona,
  \textit{Algorithms for modular elliptic curves. 2nd edition},
Cambridge Univ. Press, Cambridge (1997) MR 93m:11053.
         
         
\bibitem{C} 
J. Cremona,
\textit{Ternary cubic forms having bounded invariants, and the existence of a positive proportion of elliptic curves having rank $0$},
Annals of Mathematics 181 (2015) no. 2 587-–621.
           
   
\bibitem{D.D.P} 
S. Dasgupta, H. Darmon and R. Pollack,
\textit{Hilbert modular forms and the Gross-Stark conjecture},
Annals of Math. 174 (2011) no. 1 439--484.
   
   \bibitem{H81} 
D. R. Heath-Brown,
\textit{The fourth power mean of Dirichlet's $L$-functions},
Analysis 1  (1981) pp. 25--32.
  
  \bibitem{B1} 
D. R. Heath-Brown,
  \textit{Zero-free regions for Dirichlet $L$ functions and the least prime in an arithmetic progression},
Proc. London Math. Soc., III Ser. 64 (1992a) no. 2 265--338.
  
  \bibitem{B2} 
D. R. Heath-Brown,
  \textit{Zero-free regions of $\zeta(s)$ and $L(s, \chi)$},
Proceeding of the Amalfi conference on Nalytic Number Theory (1992b) 195--200.
       
 
\bibitem{H.J} 
 M.N Huxley and M. Jutila,
\textit{Large values of Dirichlet polynomials IV},
 Acta Arith 32 (1977) 297--312.
       
       
     \bibitem{I.M2011} 
Y. Ihara and K. Matsumoto,
     \textit{On certain mean values and the value-distribution of logarithms of Dirichlet L-functions},
Quart. J. Math. (Oxford) 62 (2011) 637--677.
          
          
    \bibitem{I.M2014} 
Y. Ihara and K. Matsumoto,
\textit{On the value-distribution of logarithms derivatives of Dirichlet L-functions},
Analytic Number Theory, Approximation Theory and Special Functions, in Honor of H. M. Srivastava, G. V. Milovanovi{\'c} and M. Th. Rassias (eds.), Springer  (2014) 79--91.
          
  \bibitem{G.Z} 
 B. H. Gross and D. B. Zagier,
\textit{Heegner points and derivatives of $L$-series},
Inventiones Mathematicae= 84 (1986) no. 2 225--320.

  \bibitem{K} 
V. A. Kolyvagin,
\textit{Finiteness of $E(\mathbb{Q})$ and $X(E, \mathbb{Q})$ for a class of Weil curves},
Math. USSR-Izvestiya 32 (1989) pp. 523--541.
    
\bibitem{Lou} 
S. Louboutin,
\textit{A twisted quadratic for Dirichlet $L$ functions},
 Proc. Amer. Math. Soc. 142 (2014) no. 5 1539--1544.

\bibitem{GP} 
The PARI Group, Bordeaux,
PARI/GP,
version 2.5.2, (2011) http://pari.math.u-bordeaux.fr/.
     
          
  \bibitem{P.R2015} 
D. Platt and O. Ramar\'e,
  \textit{Explicit estimates: from $\Lambda(n)$ in arithmetic progressions to $\Lambda(n)/n$},
yet unpublished.
     
\bibitem{R.R} 
 O. Ramar\'e and R. Rumely,
\textit{Primes in arithmetic progressions},
Mathematics of Computation 65 (1996) no. 213 397--425.
               
     
   \bibitem{R} 
K. Rubin,
\textit{The main conjectures of Iwasawa theory for imaginary quadratic fields},
Inventiones Mathematicae 103 (1991) no. 1 25--68.
    
\bibitem{R2} 
 k. Rubin,
\textit{A Stark conjecture over $\mathbb{Z}$ for abelian L-functions with multiple zeros},
 Université de Grenoble. Annales de l'Institut Fourier 46 (1996) no. 1 33–-62.
         
  \bibitem{Sound} 
 K. Soundararajan,
\textit{The fourth moment of Dirichlet $L$-functions},
Clay Mathematics Proceedings 7 (2007) 239--246.
  
  \bibitem{St1} 
H. Stark,
\textit{$L$-functions at $s=1$ I. $L$-functions for quadratic forms},
Advances in Math. 7 (1971) 301–-343.
       
\bibitem{St2} 
H. Stark,
\textit{$L$-functions at $s=1$ II. Artin $L$-functions with rational characters},
Advances in Math. 17 (1975) no. 1 60-–92.
     
\bibitem{St3} 
H. Stark,
\textit{$L$-functions at $s=1$ III. Totally real fields and Hilbert's twelfth problem},
 Advances in Math. 22 (1976) no. 1 64-–84.

\bibitem{St4} 
 H. Stark,
\textit{$L$-functions at $s=1$ VI. First derivatives at $s = 0$},
 Advances in Math. 35 (1980) no. 3 197-–235.
     
\bibitem{Ti} 
E.C. Titchmarch,
\textit{The Theory of Riemann Zeta Function},
Oxford Univ. Press, Oxford  (1951).
     
\bibitem{V} 
K. Ventullo,
\textit{On the rank one abelian Gross-Stark conjecture},
arXiv:1308.2261 (2013).
     
\bibitem{W} 
H. Walum,
\textit{An exactformula for an average of $L$-series},
IIlinois J. Math. 26 (1982) 1--3.
     
\bibitem{Y} 
 M. Young,
\textit{The fourth moment of Dirichlet $L$-functions},
Annals of Mathematics 7 (2011) 239--246.
      
\bibitem{Zh} 
 W. Zhang, \textit{Lecture Notes in Contemporary Mathematics},
 Science Press, Beijing (1989) 173--179.
     
\bibitem{Zh.W} 
W. Zhang and W. Wang,
\textit{An exact calculating formula for the $2k$-th power mean of $L$-functions},
 JP Jour. Algebra. Number Theory and Appl. 2 (2002) no. 2 195--203.

\end{thebibliography}
\end{document}